\RequirePackage{fix-cm}

\documentclass[12pt]{article}
\usepackage{a4,epsfig,amsmath,amssymb,latexsym,color,graphicx,tikz}
\usepackage{amsthm}

\def\N{\mathbb{N}}
\def\R{\mathbb{R}}

\newtheorem{theorem}{Theorem}
\newtheorem{lemma}{Lemma}
\newtheorem{proposition}{Proposition}

\newtheorem{corollary}{Corollary}
\theoremstyle{remark}\newtheorem{remark}{Remark}
\theoremstyle{remark}\newtheorem{example}{Example}

\def\to{\rightarrow}

\def\smooth#1{\mathcal C^{#1}}
\def\poly#1{\mathcal P^{#1}}

\def\spline#1#2#3{\mathcal S^{#1}_{#2,#3}}
\def\proj#1#2#3{\Pi_{#1,#2,#3}}

\def\const#1#2#3{C_{#1,#2,#3}}
\def\norm#1{\Vert#1\Vert}

\def\rt#1#2#3{R_{#1,#2,#3}}
\def\B#1#2#3{B_{#1,#2,#3}}

\def\brkH{\mathfrak H}
\def\brkS#1#2#3{\mathfrak S^{#1}_{#2,#3}}
\def\brkC#1#2#3{\mathfrak C_{#1,#2,#3}}
\def\brkR#1#2#3{\mathfrak R_{#1,#2,#3}}
\def\brkB#1#2#3{\mathfrak B_{#1,#2,#3}}

\DeclareMathOperator{\SPAN}{span}

\def\cf#1#2#3#4{C_{(#1,#4),#2,#3}}
\def\rf#1#2#3#4{R_{#1,#2,#3,#4}}
\def\Bf#1#2#3#4{B_{#1,#2,#3,#4}}

\newcommand{\Checkmark}{$\color{green}\checkmark$}

\begin{document}

\title{Approximation in FEM, DG and IGA:\\ A Theoretical Comparison}
\author{
Andrea Bressan\footnote{
Department of Mathematics,
University of Oslo, Norway}
\footnote{{\it email: andbres@math.uio.no}}
\and Espen Sande\footnotemark[1] \footnote{\it email: espsand@math.uio.no} }

%
%

\maketitle

\begin{abstract}
In this paper we compare approximation properties of degree $p$ spline spaces with different numbers of continuous derivatives. 
We prove that, for a given space dimension, $\smooth {p-1}$ splines provide better a priori error bounds for the approximation of functions in $H^{p+1}(0,1)$.
Our result holds for all practically interesting cases when comparing $\smooth {p-1}$ splines with  $\smooth {-1}$ (discontinuous) splines.
When comparing $\smooth {p-1}$ splines with $\smooth 0$ splines our proof covers almost all cases for $p\ge 3$, but we can not conclude anything for $p=2$.
The results are generalized to the approximation of functions in $H^{q+1}(0,1)$ for $q<p$, to broken Sobolev spaces and to tensor product spaces.
\end{abstract}

\section{Introduction}
The aim of this work is to compare the approximation properties of different piecewise polynomial spaces commonly employed in Galerkin methods for  PDEs. 
Following the well known Lemmas of C\'ea and Strang such approximation results imply a priori error estimates for these numerical methods. 
In particular we consider the tensor product spaces used in Discontinuous Galerkin (DG), Finite Element Methods (FEM) and IsoGeometric Analysis (IGA) that differ only in their global smoothness.
As such our comparison provides an answer to the following question: ``does smoothness impede or favour approximation?''.

It was noticed by Hughes, Cottrell and Bazilevs \cite{Hughes:2005} that smoother spaces have better approximation properties in their numerical tests.
Spline approximation in the IGA setting was first studied by Bazilevs,  Beir\~ao da Veiga, Cottrell, Hughes and Sangalli \cite{MR2250029}.
Later, Evans, Bazilevs, Babuska and Hughes \cite{Evans:2009}  numerically computed approximation constants and observed that the maximally smooth splines provide better a priori bounds on the approximation error.
Beir\~ao da Veiga, Buffa, Sangalli and Rivas \cite{MR2800710} studied how the approximation depends on the mesh-size, the degree and the global smoothness.
Takacs and Takacs \cite{Takacs:2016} proved an upper bound for the approximation error with an explicit constant.
Recently, Floater and Sande \cite{Floater:2017,Floater:2018} provided optimal constants on which we base our results.

The comparison in this paper is related to the $n$-width theory \cite{Kolmogorov:36,Pinkus:85}, i.e., looking at approximation properties of a space of fixed dimension. Our results can be seen as a partial answer to an $n$-width problem constrained to piecewise polynomial spaces on uniform partitions.
We will first look at the univariate setting before extending the results to general tensor product spaces.

Let $\poly p$ be the space of polynomials of degree at most $p$, and $L^2=L^2(0,1)$ and $H^{q+1}=H^{q+1}(0,1)$ be the standard Sobolev spaces. For a given $n\in \N$ let $I_j$ be the interval $[\frac jn,\frac{j+1}n)$ and define the spline space $\spline p k n$, of degree $p$, smoothness $k$ and on $n$  segments by
\begin{equation}\label{eq:def-spline}
\spline p k n =\big\{f \in\smooth k([0,1]): f|_{I_j} \in\poly p,\, j=0,\dots,n-1\big\}.
\end{equation}
Here $k=-1$ means that jumps are allowed at the internal breakpoints.
Furthermore, let $\proj pkn$ be the $L^2$ projection onto $\spline pkn$ and $\cf pknq$ be the smallest real number such that
\begin{equation}\label{eq:estimate-form}
\norm{f- \proj pkn f}\le \cf pknq \norm{\partial^{q+1} f}
\end{equation}
holds for all $f\in H^{q+1}$. Here $\norm{\cdot}$ denotes the $L^2$ norm.
Finally for $q=p$ we let $\const pkn:=\cf pknp$.

The studied $n$-width problem can then be formulated as follows. Given the space dimension $N$ and Sobolev regularity $q$, find the degree $p$, smoothness $k$, and number of segments $n$ such that the constant $\cf pknq$ is minimized. 
Note that for each $N$ only finitely many $(p,k,n)$ fulfill 
\begin{equation}\label{eq:spline-dim}
N=\dim\spline p k n =(n-1)(p-k)+p+1.
\end{equation}
It is then possible to try an exhaustive approach. The difficulty of such a strategy is that the constants $\cf pknq$ are solutions of  eigenvalue problems that are badly conditioned. Any conclusion based on this strategy would then have to take into consideration the reliability of the method used to compute the constant.

Our approach is to first provide lower and upper bounds for $\const pkn$ and base the conclusions on provable properties of these bounds.
In particular we compare $\const p{p-1}m$ with $\const pkn$ for $k<p-1$ under the constraint 
$$\dim\spline {p} {p-1} {m} = \dim\spline {p} {k} {n}, $$
i.e., for 
\begin{equation}\label{eq:m-formula}
m=(n-1)(p-k)+1.
\end{equation}
Note that for a fixed number of segments $n$ we have $\spline p{k_1}n \supseteq \spline p{k_2}n$ whenever $k_1\le k_2$ so that necessarily $\const p{k_1}n \le \const p{k_2}n$ under the same condition. However, for a fixed dimension $N$ the smoother space is defined on a finer partition.
We show in Section \ref{sec:comparison-1d} that $\const p{p-1}m$ is smaller than $\const pkn$ in almost all cases of practical interest for $k=0$ (see Theorem~\ref{thm:fem}) and $k=-1$ (see Theorem~\ref{thm:dg}).
In  Section~\ref{sec:low} we extend our results to the case of $p>q$ in \eqref{eq:estimate-form}. Here we compare the approximation of maximally smooth spline spaces of a ``too high degree'', $\spline {p} {p-1} {m}$, with spline spaces of lower smoothness, $\spline {q} {k} {n}$, under the constraint $\dim\spline {p} {p-1} {m} = \dim\spline {q} {k} {n}$. The main result of this section is contained in Theorem~\ref{thm:lower}. A comparison in the case of Broken Sobolev spaces is then performed in Section~\ref{sec:broken} and extensions to tensor product cases are considered in Section~\ref{sec:tens}.

The fact that smoother spaces provide better approximation could be surprising to people not familiar with the $n$-width theory, indeed it could seem reasonable that smoother spaces are more ``rigid'' and thus that they can not approximate functions that are less regular. 
This is not true: for instance it was shown by Kolmogorov \cite{Kolmogorov:36} that
$$
\SPAN \{1,\cos(\pi x),\ldots,\cos((N-1)\pi x)\} 
$$
is optimal for $H^1$ in the $n$-width sense, meaning that no other $N$-dimensional subspace of $L^2$ can provide a better a priori error estimate for $H^1$ functions.
Based on results of Melkman and Micchelli \cite{Melkman:78} it was proved in \cite{Floater:2017} that for all $q$ and $N$ there exists an optimal $\smooth {\ell (q+1)-2}$ spline space of degree $\ell (q+1)-1$ for any $\ell=1,2,\dots$.  In fact, for $q=0$ and with even degrees $p$, the knots of the optimal spline spaces are uniform and so they are subspaces of $\spline p{p-1}n$.

\section{Upper and lower bounds for $\const pkn$}\label{sec:univ}

We prove the following bounds on the best constants $\const {p}{k}n$, $k=-1,0, \ldots, p-1$.
\begin{theorem}\label{thm:bounds}
For all $p\ge 0$, $k\in\{-1,0,\ldots,p-1\}$ and $n\ge 1$ we have
\begin{align}
\label{est:-1}   \frac{(p+1)!}{(2p+2)!\sqrt{2p+3}} n^{-p-1}\le &\const {p}{k}n \le (n\pi)^{-p-1}
\end{align}
\end{theorem}
The above inequalities are shown in the following lemmas.

\begin{lemma}\label{lem:legendre-poly} For discontinuous spline approximation we have
$$
 \frac{(p+1)!}{(2p+2)!\sqrt{2p+3}}n^{-p-1} \le \const p{-1}n.
$$
\end{lemma}

\begin{proof}
It is enough to show that there exists an $f\in H^{p+1}$ such that 
$$
\norm{f-\proj p{-1}n f}\ge \frac{(p+1)!}{(2p+2)!\sqrt{2p+3}} n^{-p-1}  \norm{\partial^{p+1}f}.
$$
This is the case for $f(x)=x^{p+1}$. 
The projection $\proj p{-1}n$ acts independently on each $I_j$, $j=0,\dots,n-1$ and on $I_j$ we have 
$$
x^{p+1}= \sum_{i=0}^{p+1} c_{i,j}\,\ell_i (n x-j ) ,$$
where $\ell_i$ is the $i$-th Legendre polynomial on $[0,1]$.
Since $\ell_{p+1}(nx-j)$ is orthogonal to the polynomials of degree $p$ on $I_j$ we have
$$x^{p+1} -\proj p{-1}n x^{p+1}  =  \sum_{j=0}^{n-1} c_{p+1,j} \ell_{p+1}( n x-j ) .
$$
Since $$
\norm{\ell_{p+1}}=(\sqrt{2p+3})^{-1}\qquad \text{and}\qquad \norm{\partial^{p+1}\ell_{p+1}}=\frac{(2p+2)!}{(p+1)!}.
$$
 by taking the derivative of $\ell_{p+1}( n x-j )$ we obtain
$$
\norm{x^{p+1}-\proj p{-1}n x^{p+1}}_{I_j}=  \frac{(p+1)!}{(2p+2)!\sqrt{2p+3}}n^{-p-1} \norm{\partial^{p+1} x^{p+1}}_{I_j}.
$$
Summing over $j$ the squares of the left and right hand sides yields the result.
\end{proof}

\begin{lemma}\label{lem:univk}
For maximally smooth spline approximation we have
$$\const {p}{p-1}n \le (n\pi)^{-p-1}.$$
\end{lemma}

\begin{proof}
This is a corollary of the results in \cite{Floater:2018}.
Let 
$$E^{p}=\{f\in H^{p+1}:\quad \partial^{s} f(0)=\partial^{s} f(1)=0,\quad 0\leq s<  p,\quad s+p\text{ is  odd}  \}.$$
Observe that for $p$ odd, $E^p$ coincides with the non-standard Sobolev space $H^{p+1}_0$ defined in \cite{Floater:2018}, and for $p$ even it coincides with the space $H^{p+1}_1$ in that paper.

Then \cite[Theorems 1 and 2]{Floater:2018} states that for all $v\in E^p$
\begin{equation}\label{eq:Floater-result}
\norm{v-\Pi_E v}\le \Big(\frac{1}{n\pi}\Big)^{p+1}\norm{\partial^{p+1}v},
\end{equation}
where $\Pi_E :E^p\to E^p\cap \spline p{p-1}n$ is the $L^2$ projection.
Note that the $n$ in \cite{Floater:2018} is the space dimension, what we call $N$, and not the number of segments.

Given $f\in H^{p+1}$ let $g\in \poly p$ be a polynomial such that $f-g\in E^p$. In other words, for $0\le s < p$ with $s+p$ odd, we have
$$
\left\{\begin{aligned}
\partial^s g(0)&=\partial^s f(0)\\
\partial^s g(1)&=\partial^s f(1).
\end{aligned}\right .
$$
This $g$ exists according to Lemmas \ref{lem:poly-interp} and \ref{lem:poly-interp2} in the appendix.
Then, since $g\in\spline p{p-1}n$ and $f-g\in E^p$ we have
$$
\begin{aligned}
\norm{f-\proj p{p-1}n f}&= \norm{(f-g) - \proj p{p-1}n  (f-g)}\\
	&\le \norm{(f-g) - \Pi_E (f-g)}\\
	&\le  \Big(\frac{1}{n\pi}\Big)^{p+1}\norm{\partial^{p+1}(f-g)}\\
	&= \Big(\frac{1}{n\pi}\Big)^{p+1}\norm{\partial^{p+1}f}.
\end{aligned}
$$
\end{proof}

Theorem \ref{thm:bounds} now follows from the observation that $\spline p {k+1}n\subset\spline p {k}n$ for all $k=-1,0,\ldots,p-2$ and so $\const pkn\le \const p{k+1}n$.

\section{Univariate comparisons}
\label{sec:comparison-1d}

Here we compare the space of maximally smooth splines, $\spline p{p-1}m$, commonly used in IGA, with the space $\spline pkn$ of smoothness $k<p-1$ where $m$ depends on $n$ as in \eqref{eq:m-formula}, i.e., such that $\dim\spline p{p-1}m=\dim\spline pkn$. This means that the smoother space is defined on a finer grid. See Figure \ref{fig:deg1-3} for an example of this. Note that the case $k=p-1$ and the case $n=1$ are uninteresting since we would then be comparing $\spline pkn$ with itself.

The estimates in Section \ref{sec:univ} are sharp enough to prove that smooth splines will eventually provide a better approximation in the number of degrees of freedom.
This is stated in the following theorem.
More precise statements for the IGA-FEM comparison ($k=0$) and the IGA-DG comparison ($k=-1$) are contained in subsections 3.1 and 3.2.

\begin{figure}
\includegraphics{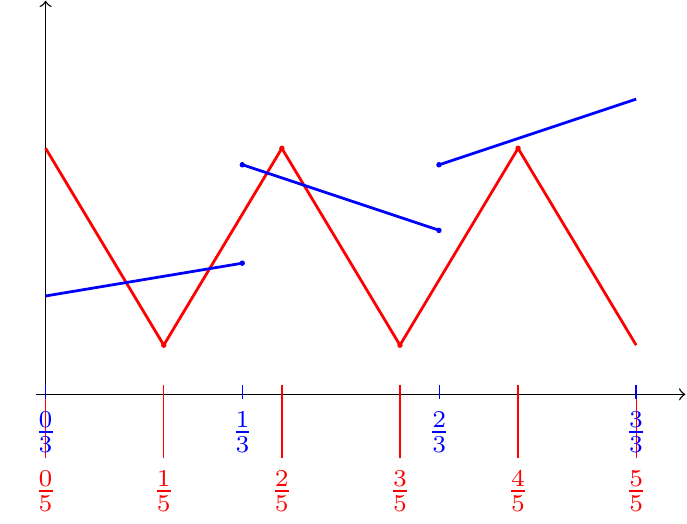} \includegraphics{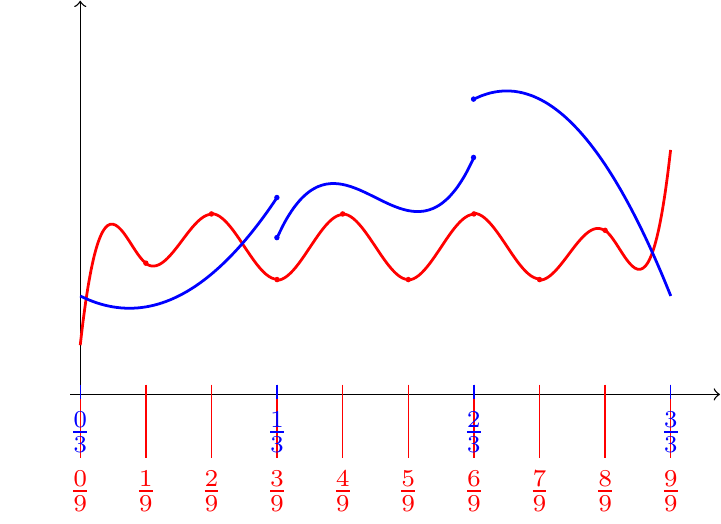}
\caption{Example of pairs of functions in $\spline p{-1}n$ (blue) and $\spline p{p-1}m$ (red) for $p=1$ and $p=3$, $n=3$. Note how the maximally smooth spline space is defined on a finer grid.}\label{fig:deg1-3}
\end{figure}

\begin{theorem}\label{thm:comparison-general}
For all $k\geq -1$ and $n\ge 2$ there exists $\bar p$ such that for all $p\ge \bar p$  
$$
\const p{p-1}m<\const pkn,
$$
where $m=(n-1)(p-k)+1$.
\end{theorem}

This theorem follows from studying the bounds in Theorem \ref{thm:bounds}, which is done in Lemma \ref{thm:ratio} and Proposition \ref{prop:B}.

\begin{lemma}\label{thm:ratio}
For all $p\ge 0$, $k\in\{-1,\dots,p-1\}$ and $n,m\ge 1$ we have
\begin{equation}\label{eq:ratio}
\frac{\const p{p-1}m}{\const pkn} \le \Big(\frac{4}{e\pi}\Big)^{p+1}\Big(\frac{n}{m}\Big)^{p+1} (p+1)^{p+1}\sqrt{4p+6}.
\end{equation}
\end{lemma}

\begin{proof}
From Theorem \ref{thm:bounds} we have for $k=-1,\dots,p-1$, that
\begin{equation*}
\frac{\const p{p-1} m}{\const pkn} \le \Big(\frac{n}{m\pi}\Big)^{p+1} \frac{(2p+2)!}{(p+1)!}\sqrt{2p+3}.
\end{equation*}
Now, using the error bounds of the Stirling's approximation~\cite{Robbins:55}
\begin{equation}\label{ineq:Stirling}
\sqrt{2\pi}r^{r+\frac{1}{2}}e^{-r}e^{\frac{1}{12r+1}}\le r! \le \sqrt{2\pi}r^{r+\frac{1}{2}}e^{-r}e^{\frac{1}{12r}},
\end{equation}
we find that 
\begin{align*}
\frac{(2p+2)!}{(p+1)!} &\le \frac{\sqrt{2\pi}(2p+2)^{2(p+1)+\frac{1}{2}}e^{-2(p+1)}e^{\frac{1}{12(2p+2)}}}{\sqrt{2\pi}(p+1)^{p+1+\frac{1}{2}}e^{-p-1}e^{\frac{1}{12(p+1)+1}}},\\
&=2^{2(p+1)+\frac{1}{2}}
\frac{(p+1)^{2(p+1)+\frac{1}{2}}}{(p+1)^{p+1+\frac{1}{2}}}
\frac{e^{-2(p+1)}}{e^{-p-1}}
\frac{e^{\frac{1}{12(2p+2)}}}{e^{\frac{1}{12(p+1)+1}}},\\
&=4^{p+1}\sqrt{2}(p+1)^{p+1} e^{-p-1}
e^{\frac{1}{24(p+1)}-\frac{1}{12(p+1)+1}},\\
&=\Big(\frac{4}{e}\Big)^{p+1}\sqrt{2}(p+1)^{p+1}e^{\frac{1-12(p+1)}{24(p+1)(1+12(p+1))}},\\
&\le \Big(\frac{4}{e}\Big)^{p+1}\sqrt{2}(p+1)^{p+1},
\end{align*}
and the result follows.
\end{proof}

For $m$ as in  \eqref{eq:m-formula}, we let $\rt pnk$ be the bound in \eqref{eq:ratio}, now given as
\begin{align}\label{eq:ratio-m}
\rt pkn=(\B pkn)^{p+1}  \sqrt{4p+6}\qquad\text{with}\\
\label{eq:base-m}
\B pkn = \frac{4}{e\pi}  \frac{n (p+1)}{(p-k)(n-1)+1} .
\end{align}
The next step of our analysis is the study of $\B pkn$.

\begin{proposition}\label{prop:B}
For $-1\leq k <p-1$ and $n\ge 2$  the following statements hold
\begin{enumerate}
\item \label{p-lim} for fixed $n$ and $k $  $$\lim_{p\to\infty} B_{p,k,n}= \frac {4 } {e\pi}\frac n{n-1}<1;$$
\item \label{n-lim} for fixed $p$ and $k$, $$\lim_{n\to\infty} B_{p,k,n}= \frac {4 } {e\pi }\frac{p+1}{p-k};$$
\item\label{n-decreasing} $B_{p,k,n}$ is strictly decreasing in $n$;
\item\label{p-decreasing} $B_{p,k,n}$ is decreasing in $p$ for $k\ge 0$.
\end{enumerate}
\end{proposition}
\begin{proof}
The limits in points \ref{p-lim} and \ref{n-lim} are straightforward.

For \ref{n-decreasing} it is sufficient to show that $\B p{n+1}k<\B {p}nk$, i.e.,
\begin{align*}
&\frac{n+1}{(p-k)n+ 1}<\frac{n}{(p-k)(n-1) + 1},
\end{align*}
which is equivalent to $k<p-1$, since the denominators are positive.

For \ref{p-decreasing} we prove that $\B {p+1}nk \leq \B pnk$, i.e.,
\begin{align*}
 \frac{p+2}{(p-k+1)(n-1)+1} \leq \frac{p+1}{(p-k)(n-1)+1}.
\end{align*}
This is equivalent to $(k+1)(n-1)\geq 1$, which holds for $n\geq 2$ and $k\geq 0$.
\end{proof}

\begin{proof}[Proof of Theorem \ref{thm:comparison-general}]
From point \ref{p-lim} of Proposition \ref{prop:B} we deduce that for $p>\hat p$ we have $\B pkn \le t<1$ and
$$
\rt pkn = (\B pkn)^{p+1}\sqrt{4p+6}\le t^{p+1}\sqrt{4p+6}.
$$
Thus there exists $\bar p\ge \hat p$ such that for all $p>\bar p$, $\rt pkn <1$.
\end{proof}

\begin{remark}\label{rem:exp}
Using Proposition \ref{prop:B} we can obtain an estimate of how much better  the approximation with smooth splines is in Theorem \ref{thm:comparison-general}.
For a fixed $k$, and any $\bar p,\bar n$ satisfying $\B {\bar p}k{\bar n}\leq \frac{4}{e\pi}\gamma<1$ we have
\begin{equation}\label{ineq:exp}
\rt pkn \leq \Big ( \frac 4{e\pi}\gamma \Big)^{p+1} \sqrt{4p+6},\qquad \forall n\geq\bar n, \, \forall p\geq\bar p,
\end{equation}
i.e. $\rt pkn$ gets exponentially smaller as $p$ increases.
The level set $\B pkn=\frac{4}{e\pi}\gamma$ is the hyperbola
\begin{align*}
0 = \big(n -\frac {\gamma}{\gamma-1} \big)\big(p-\frac{\gamma k+1}{\gamma-1} \big) + \frac{\gamma(\gamma-k-2)}{(\gamma-1)^2}
\end{align*}
and has asymptotes
\begin{align*}
p= \frac{\gamma k+1}{\gamma-1}, \qquad n=\frac \gamma{\gamma-1}.
\end{align*}
This tells us that for each $\bar n\geq \frac \gamma{\gamma-1}$ there is a corresponding $\bar p$ such that we obtain the exponential improvement in \eqref{ineq:exp}. 
\end{remark}

\begin{corollary}\label{thm:decreasing-n}
For all $p\ge1$ and $k=-1,\dots,p-2$,  the ratio $\rt pkn$ in \eqref{eq:ratio-m} is strictly decreasing in $n$.
\end{corollary}

\begin{proof}
By definition
$$
\rt pkn=(\B pkn)^{p+1}  \sqrt{4p+6}
$$
and $\B pkn^{p+1}$ is strictly decreasing in $n$ by point \ref{n-decreasing} of Proposition \ref{prop:B}.
\end{proof}

\begin{corollary}\label{thm:decreasing-p}
For all $k\ge 0$,  $\rt pkn$ is strictly decreasing in $p$ for all $p\ge \bar p$ where $\bar p$ is such that $\rt {\bar p}kn\le1$.
\end{corollary}

\begin{proof}
From point \ref{p-decreasing} of Proposition \ref{prop:B}, $\B pkn$ is decreasing in $p$. Moreover, $(4p+6)^{1/(2p+2)}$ is strictly decreasing in $p$. Thus $(\rt pkn)^{1/(p+1)}$ is also strictly decreasing in $p$ and the result follows.
\end{proof}

\begin{remark}\label{rem:quarter}
For fixed $k\geq 0$ and given $\bar p$ and $\bar n$ such that $\rt{\bar p}k{\bar n}<1$ then from the above corollaries we find that 
  $$\const   p{  p-1} m < \const   p{k}{  n},\qquad \forall p\ge \bar p,\quad \forall n \ge \bar n.$$
  This means that there cannot be isolated values for which this inequality holds.
A similar result is true for $k=-1$, but it requires a more technical argument that is postponed until later.
\end{remark}

\subsection{IGA-FEM comparison}

\begin{theorem}[IGA-FEM comparison]\label{thm:fem}
For $p\geq 3$ and sufficiently large $n$, more precisely
$$
\begin{aligned}
&n\geq 7   &&\text{for} &&p=3,
\\&n\geq 4 &&\text{for} &&p\in\{4,5\},
\\&n\geq 3 &&\text{for} &&p\in\{6,...,37\},
\\&n\geq 2 &&\text{for} &&p\ge 38,
\end{aligned}
$$
we have 
$$
\const p{p-1}m<\const p0n.
$$
\end{theorem}

Note that for $n=1$ or $p=0$ the spaces are the same and $\const p{p-1}m=\const p0n$.
Note further that no conclusion can be drawn for $p=2$. Indeed we have
$$
\rt 20n=\Big(\frac {4}{e\pi}\frac{3n}{2n-1}\Big)^3 \sqrt{14} >\Big(\frac {6}{e\pi}\Big)^3 \sqrt{14}> 1, \quad \forall n\ge 2.
$$
All cases are summarized in Fig.~\ref{fig:iga-fem}.

\begin{proof}
Using Corollary \ref{thm:decreasing-n} and Corollary \ref{thm:decreasing-p} as explained in Remark \ref{rem:quarter} it is enough to show that $\rt {38}0{2}$,  $\rt 603$, $\rt 404$ and  $\rt 307$  are less than $1$.
We have
\begin{align*}
\rt {38}02 &= \Big(\frac8{e\pi }\Big)^{39}\sqrt{158} = 0.9851\ldots\\
\rt 603&= \Big(\frac4{e\pi }\frac {21}{13} \Big)^{7}\sqrt{30} = 0.7776\ldots\\
\rt 404&= \Big(\frac4{e\pi }\frac {20}{13} \Big)^{5}\sqrt{22} = 0.9114\ldots\\
\rt 307&= \Big(\frac4{e\pi }\frac {28}{19} \Big)^{4}\sqrt{18} = 0.9632\ldots
\end{align*}
\end{proof}

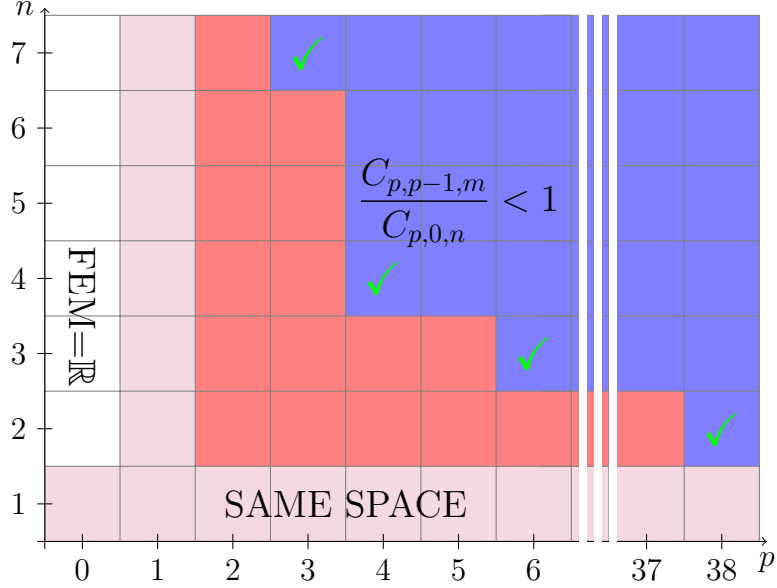
\begin{figure}
\begin{center}
\begin{tikzpicture} 

\fill[purple!15]  (0,0)--(0,1)--(1,1)--(1,7)--(2,7)--(2,1) --(9.5,1)--(9.5,0);
\fill[white] (0,1) rectangle (1,6.8);
\fill[blue!50] (3,6)--(4,6)--(4,3)--(6,3)--(6,2)--(8.5,2)--(8.5,1)--(9.5,1)--(9.5,7)--(3,7);
\fill[red!50] (2,1) -- (8.5,1)--(8.5,2)--(6,2)--(6,3)--(4,3)--(4,6)--(3,6)--(3,7)--(2,7);

\draw[gray] (0,0) grid (7,7);
\begin{scope}[xshift=7.5cm]
\draw[gray] (-.9,0) grid (2,7);
\end{scope}

\draw[->] (-.1,0)--(9.6,0) node[below] {$p$};
\draw[->] (0,-.1)--(0,7.1) node[left]      {$n$};

\node at (4,.5) {\large SAME SPACE};
\node[rotate=-90] at (.5,3) {\large FEM=$\R$};

\foreach \x[evaluate=\x as \xp using \x+.1] in {7.1,7.3,7.5} { \fill[white] (\x,-.1) rectangle  (\xp,7.1); }

\foreach \p[evaluate=\p as \l using \p+.5] in {0,1,...,6} {\draw (\l,.1)--(\l,-.1) node[below] {$\p$};}
\foreach \n[evaluate=\n as \l using \n-.5] in {1,2,...,7} {\draw (.1,\l)--(-.1,\l) node[left]   {$\n$};}
\foreach \p/\l in {37/8, 38/9} {\draw (\l,.1)--(\l,-.1) node[below] {$\p$};}

\node at (5.5,4.5) {\large $\displaystyle\frac {\const p{p-1}m}{\const p0n}<1$};

\node at (3.5,6.5) {\Large \Checkmark};
\node at (4.5,3.5) {\Large \Checkmark};
\node at (6.5,2.5) {\Large \Checkmark};
\node at (9,1.5) {\Large \Checkmark};

\end{tikzpicture}
\end{center}
\caption{The blue area indicates for which $p$ and $n$ we can conclude that IGA approximation is better than FEM approximation. The red area indicates where no conclusion can be obtained from the estimate. The two spaces coincide in the pink area.}\label{fig:iga-fem}
\end{figure}

\subsection{IGA-DG comparison}\label{sec:dg}

Similarly to the previous subsection we note that for $n=1$ or $p=0$ the spaces are the same and $\const p{p-1}m=\const p{-1}n$.

\begin{lemma}\label{thm:decreasing-p-1}
For $n=2$, $p\ge 22$ and $n=3$, $p\ge 2$ the function $\rt  p{-1}n$ is strictly decreasing in $p$.
\end{lemma}

\begin{proof}
First note that $\rt  p{-1}n$ is decreasing whenever $\rt  p{-1}n^2$ is decreasing.
We now let $s=p+1$ and compute the derivative of $\rt  {s-1}{-1}n^2$ with respect to $s$ and show that it is negative.
$$
\begin{aligned}
\partial_s&( \rt  {s-1}{-1}n)^2= \underbrace{\frac 4{ns-s+1}\Big(\frac 4{e\pi} \frac{ns}{ns-s+1}\Big )^{2s} }_{> 0}\\
&\Big(  1-2 s^2 (n-1) +  \underbrace{(1+2s)(ns-s+1)}_{>0} \underbrace{\ln \big(\frac 4{\pi} \frac {ns}{ns-s+1} \big)}_{\leq L} \Big),
\end{aligned}
$$
where $L=\ln \big(\frac 4{\pi} \frac {n}{n-1} \big)<1$ is an upper bound on the logarithm.
It follows that $\partial_s \rt  {s-1}{-1}n <0$ if 
$$
2(n-1)(L-1)s^2 + (n+1)L s + (L+1)<0,
$$
i.e., for 
$$
s> \frac { -(n+1)L -\sqrt{(n+1)^2L^2-8(L^2-1)(n-1)} }{4(n-1)(L-1)}.
$$
For $n=2$ we have $L=\ln\frac 8\pi < 0.935$ and $\rt p{-1}2$ is strictly decreasing for
$$
p=s-1\ge \frac {3L+\sqrt{L^2+8}}{4(1-L)}-1\approx 21.14\ldots.
$$
For $n=3$ we have $L=\ln\frac 6\pi < 0.648$ and $\rt p{-1}3$ is strictly decreasing for
$$
p=s-1\ge\frac12\frac{1+L}{1-L}-1\approx 1.33\ldots.
$$ 
\end{proof}

\begin{theorem}[IGA-DG comparison]\label{thm:dg}
	For$$
	\begin{aligned}
		&n\ge 3 &&\text{and}&&p\in\{1,\dots,17\},
		\\&n\ge 2 &&\text{and}&&p\ge 18,
	\end{aligned}$$
we have 
$$
\const p{p-1}m<\const p{-1}n.
$$
\end{theorem}

\begin{proof}
Using Lemma \ref{thm:decreasing-n} and Lemma \ref{thm:decreasing-p-1} it is enough to show that $\rt 1{-1}3$, $\rt 2{-1}3$ and $\rt {22}{-1}2$ are less than $1$ to cover all cases but $\rt {18}{-1}2$, $\rt {19}{-1}2$, $\rt {20}{-1}2$, $\rt {21}{-1}2$. The latter are also checked.
We have
\begin{align*}
\rt 1{-1}3&=\Big(\frac4{e\pi }\frac 65 \Big)^{2}\sqrt{10} = 0.9990\ldots\\
\rt 2{-1}3&=\Big(\frac {4}{e\pi }\frac 97 \Big)^{3}\sqrt{14} = 0.8172\ldots\\
\rt {18}{-1}2&=\Big(\frac4{e\pi }\frac {19}{10} \Big)^{19}\sqrt{ 78} =  0.9639 \ldots\\
\rt {19}{-1}2&=\Big(\frac4{e\pi }\frac {40}{21} \Big)^{20}\sqrt{ 82} =  0.9247 \ldots\\
\rt {20}{-1}2&=\Big(\frac4{e\pi }\frac {21}{11} \Big)^{21}\sqrt{ 86} =  0.8862 \ldots\\
\rt {21}{-1}2&=\Big(\frac4{e\pi }\frac {44}{23} \Big)^{22}\sqrt{ 90} =  0.8484 \ldots\\
\rt {22}{-1}2&=\Big(\frac4{e\pi }\frac {23}{12} \Big)^{23}\sqrt{ 94} =  0.8115 \ldots
\end{align*}
\end{proof}
Note that nothing can be concluded for $n=2$ and $p\in\{1,\dots,17\}$ since the estimate $\rt p{-1}2>1$ in these cases. 
All cases are summarized in Fig.~\ref{fig:iga-dg}.

\begin{figure}
\begin{center}

\begin{tikzpicture} 
\fill[purple!15]  (0,0)--(0,7)--(1,7)--(1,1)--(12,1)--(12,0);
\fill[blue!50] (1,2)--(6,2)--(6,1)--(12,1)--(12,7)--(1,7);
\fill[red!50] (1,1) rectangle(6,2);

\draw[gray] (0,0) grid (12,7);

\draw[->] (-.1,0)--(12.1,0) node[below] {$p$};
\draw[->] (0,-.1)--(0,7.1) node[left]      {$n$};

\node at (2,.5) {\large SAME SPACE};
\node at (8.5,4) {\large $\displaystyle\frac {\const p{p-1}m}{\const p{-1}n}<1$};

\foreach \p[evaluate=\p as \l using \p+.5] in {0,1,...,3} {\draw (\l,.1)--(\l,-.1) node[below] {$\p$};}
\foreach \n[evaluate=\n as \l using \n-.5] in {1,2,...,6} {\draw (.1,\l)--(-.1,\l) node[left]   {$\n$};}

\foreach \x[evaluate=\x as \xp using \x+.1] in {4.3,4.5,4.7} { \fill[white] (\x,-.1) rectangle  (\xp,7.1); }
\foreach \p/\l in {17/5.5,18/6.5,19/7.5,20/8.5,21/9.5,22/10.5,23/11.5} {\draw (\l,.1)--(\l,-.1) node[below] {$\p$};}

\node at (1.5,2.5) {\Large \Checkmark};
\node at (2.5,2.5) {\Large \Checkmark};
\node at (6.5,1.5) {\Large \Checkmark};
\node at (7.5,1.5) {\Large \Checkmark};
\node at (8.5,1.5) {\Large \Checkmark};
\node at (9.5,1.5) {\Large \Checkmark};
\node at (10.5,1.5) {\Large \Checkmark};

\end{tikzpicture}
\end{center}
\caption{The blue area indicates for which $p$ and $n$ we can conclude that IGA approximation is better than DG approximation. The red area indicates where no conclusion can be obtained from the estimate.  The two spaces coincide in the pink area.}\label{fig:iga-dg}
\end{figure}
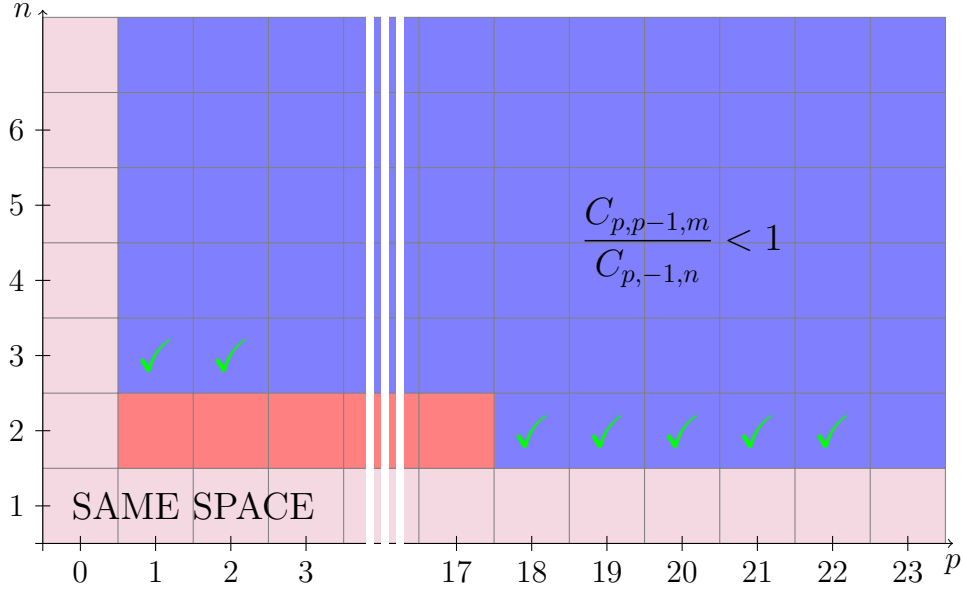

\section{Lower order Sobolev spaces}
\label{sec:low}

In this section we consider an approximand $f$ in $H^{q+1}$ and compare the approximation by smooth splines of degree $p>q$, $\spline p{p-1}m$, with that by $\smooth k$ splines of degree $q$, $\spline qkn$.
Both spaces provide the same approximation order, but the smoother space has a degree higher than the regularity of the approximand.
In IGA the degree of the spline space is sometimes  determined by the parametrisation of the domain, and not by the Sobolev regularity of the solution.
Our aim is to show that, for practical purposes, smooth spline spaces of degree higher than the Sobolev regularity have better approximation properties than lower smoothness spaces of the optimal degree.

Recalling that $\cf pknq$, $0\le q\le p$ is the best constant in 
$$
\norm {f-\proj pkn f}\le \cf pknq \norm{\partial^{q+1}f},
$$
we compare $\cf p{p-1} m q$ with $\const q kn $ under the constraint $\dim \spline p{p-1}m = \dim \spline qkn,$
which corresponds to
\begin{equation}\label{eq:m-formula-low}
m=(q-k)(n-1)+1+q-p.
\end{equation}

\begin{theorem}\label{thm:lower-norm}
For all $0\le q \le p$, $k\in\{-1,0,\ldots,p-1\}$ and $n\ge 1$ we have
$$
\cf pknq \le \Big(\frac1{n\pi}\Big)^{q+1}.
$$
\end{theorem}

The proof is done only for $k=p-1$ and using induction starting from $q=0$.
The base case is proved in the following lemma.

\begin{lemma}\label{lem:ineqH1}
For all $p\geq 0$ and $n\geq 1$ we have
\begin{equation}\label{ineq:H1}
\cf p{p-1}n0 \le \frac1{n\pi}.
\end{equation}
\end{lemma}
\begin{proof}
If $p$ is even, then this follows directly from \cite[Theorem~2]{Floater:2018} (originally shown in \cite[Theorem~2]{Floater:2017}) where it is proved for a subspace of $\spline p{p-1}{n}$. 

If $p$ is odd, \cite[Theorem~2]{Floater:2018} states approximation results for splines on a different partition.
We obtain the desired result by extending the domain to 
$$
\widetilde I=(-\frac{1}{2n},1+\frac{1}{2n})
$$
and considering the spaces 
\begin{align*}
\widetilde E&=\{ f\in \smooth {p-1}(\widetilde I) : \partial^s f(-\frac 1{2n} )= \partial^s f(1+\frac 1{2n} )=0, \ 1\le s\le p,\,s\text{ odd} \}\\
\widetilde{\mathcal S}&=\{f\in \widetilde E  :  f|_{[-\frac 1{2n},0)},\,  f|_{[1,1+\frac 1{2n})},\,f|_{I_j}\in\poly p, \ j=0,\ldots,n-1 \}
\end{align*}
where we recall $I_j=[\frac jn,\frac{j+1}n)$. 
Note that $\spline pkn$ is the restriction of $\widetilde{\mathcal S}$ to $[0,1]$ and that $\dim  \widetilde {\mathcal S}=n+1$.
Furthermore let $\widetilde \Pi: L^2(\widetilde I) \to  \widetilde {\mathcal S} $ be the orthogonal projection  and $\mathcal{E}: H^1(I)\to H^1(\widetilde I)$ be the extension operator 
$$
\mathcal{E} f(x) =\begin{cases}
	f(0)& x\le 0,\\
	f(x)& 0<x\le 1,\\
	f(1)& x>1.
\end{cases}
$$
Then using \cite[Theorem~2]{Floater:2018} on $\widetilde I$  we get
\begin{align*}
\norm{f-\proj p{p-1}n f}&\le \norm{f- (\widetilde \Pi \circ \mathcal{E} f)|_{I}}
\le  \norm{\mathcal{E} f-  \widetilde \Pi \circ \mathcal{E} f}_{\widetilde I}\\
&\le  \frac{n+1}{n}\frac 1{(n+1)\pi}\norm{\partial \mathcal{E} f}_{\widetilde I} = \frac{1}{n\pi}\norm{\partial f},
\end{align*}
where the factor $(n+1)/n$ is the length of $\widetilde I$, $n+1$ is $\dim \widetilde{\mathcal S}$ and $\norm{\cdot}_{\widetilde I}$ denotes the $L^2$ norm on the interval $\widetilde I$.
\end{proof}

\begin{proof}[Proof of Theorem \ref{thm:lower-norm}]
The proof is by induction. The cases $p\ge q=0$ are proved in Lemma \ref{lem:ineqH1}.
The case $(p,q)$ is proved assuming  the result is true for $(p-1,q-1)$, namely that for $f\in H^q$ we have
$$\norm {f - \proj {p-1}{p-2}n f }\le \Big(\frac 1 {n\pi}\Big)^{q}\norm{\partial^q f}.$$
Let $Q:H^{1}\to \spline p{p-1}n$ be the projection defined by
\begin{equation}\label{eq:H1proj}
Q f (x) =c(f) + \int_0^x \proj {p-1}{p-2}n \partial f(z)\,dz
\end{equation}
 where $c(f)\in\R$ is uniquely determined by requiring that $Q$ is a projection.
Since $\proj p{p-1}n$ is a projection and using Lemma \ref{lem:ineqH1}, we have for $f\in H^{q+1}$
\begin{align*}
\norm{f-\proj p{p-1}n f }&= \norm{ (f-Q f) -\proj p{p-1}n ( f-Q f) }
\le \Big(\frac 1 {n\pi}\Big) \norm{\partial (f-Q f)}. 
\end{align*}
Using 
\eqref{eq:H1proj} and the induction hypothesis on  $\partial f\in H^q$ we obtain
\begin{align*}
\Big(\frac 1 {n\pi}\Big) \norm{\partial (f-Q f)} = \Big(\frac 1 {n\pi}\Big) \norm{\partial f - \proj {p-1}{p-2}n \partial f} \le \Big(\frac 1 {n\pi}\Big)^{q+1} \norm{\partial^{q+1}f}.
\end{align*}
\end{proof}

\begin{theorem}\label{thm:lower}
 Let $q$ and $k<q-1$ be given. If $\rt qk{\bar n}<1$ for some $\bar n$, then for all $p\ge q$,
\begin{equation}\label{ineq:lower}
n\ge \frac{p-k-1}{q-k-1}\bar n
\end{equation}
and with $m$ as in \eqref{eq:m-formula-low}, it holds
$$
\cf p{p-1}mq<\const qkn.
$$
\end{theorem}
Observe that for fixed $n, k$ and $q$, the degree $p$ can only be increased until it reaches $p=(q-k)(n-1) +q $.
At that point $m=1$ and $\spline p{p-1}m=\poly p$.

\begin{proof}
	Similar to Section \ref{sec:comparison-1d} we have
	$$
	\frac {\cf p{p-1}mq}{\const qkn}\le \rf pknq
	$$
	where 
	\begin{align*}
	\rf pknq&= (\Bf pknq)^{q+1}\sqrt{4q+6}\quad\text{with}\\
	\label{eq:bf-pknq}\Bf pknq&=\frac{4}{e\pi} \frac{n (q+1)}{(q-k)(n-1)+1+q-p}.
	\end{align*}	
Moreover, $$\Bf pknq \le \B qk{\bar n}\quad \Rightarrow\quad\rf pknq\le \rt qk{\bar n}$$ 
and
 $\Bf pknq \le \B qk{\bar n}$ is equivalent to \eqref{ineq:lower}.
\end{proof}

Comparing $\Bf pknq$ in the above proof with $\B pkn$ for the case  $p=q$ in Section~\ref{sec:comparison-1d}, there is only an additional $q-p$ in the denominator.

\begin{example}
IGA-DG comparison in $H^2$. It follows from Theorem \ref{thm:dg} that for $k=-1$ and $q=1$ we can choose $\bar n=3$ in \eqref{ineq:lower}. Thus the IGA space of degree $p\geq 1$ gives better approximation in $H^2$ than the DG space of degree $1$ for all $n\geq 3p$. 
\end{example}

\begin{example}
IGA-FEM comparison in $H^4$. It follows from Theorem \ref{thm:fem} that for $k=0$ and $q=3$ we can choose $\bar n=7$ in \eqref{ineq:lower}. Thus the IGA space of degree $p\geq 3$ gives better approximation in $H^4$ than the FEM space of degree $3$ for all $n\geq 7(p-1)/2$. 
\end{example}

\section{Broken Sobolev spaces}
\label{sec:broken}

In numerical methods for PDEs, and especially in IGA \cite{Buffa:14}, it is common to consider broken Sobolev spaces, i.e., spaces of functions that are piecewise in $H^{p+1}$. 
The problem of approximating functions in broken Sobolev spaces arises in DG, PDEs with discontinuous coefficients and in isoparametric methods where the parametrization is only piecewise regular.
The aim of this section is to show that smooth spline spaces have better approximation properties, provided the discontinuities are representable in the spline space and that the partitions are sufficiently fine.

Let $\Xi=\{\xi_1<\dots<\xi_T\}\subset (0,1)$ be a set of breakpoints and $S=(s_1,\dots,s_{T})$, $s_i\in\{-1,\dots,p-1\}$,  be the corresponding smoothness parameters. For notational reasons let $\xi_0=0$ and $\xi_{T+1}=1$.
We consider the broken Sobolev space $\brkH^{p+1}(\Xi,S)$ defined by
\begin{equation}
\brkH^{p+1}(\Xi,S) = \left\{
\begin{aligned}f:\ &\partial^{p+1} f \in L^2(\xi_i,\xi_{i+1}),\, i=0,\dots,T
\\& \partial^{s_i+1}  f \in L^2(\xi_{i-1},\xi_{i+1}),\,i=1,\dots,T
\end{aligned} \right\}.
\end{equation}
See Figure \ref{fig:broken} for an example.
We will consider error estimates of the type $$
\norm{f-\Pi f} \le C \norm{\partial^{p+1}f}_\Xi
$$
where $\norm{\cdot}_\Xi$ is the piecewise $L^2$ norm:
$$
\norm{g}_\Xi = \Big (\sum_{i=0}^T \norm{ g}^2_{L^2(\xi_i,\xi_{i+1})} \Big)^{\frac 12}.
$$

To achieve the expected approximation order it is necessary that the spline spaces can represent the discontinuities of the derivatives of the functions in $\brkH^{p+1}(\Xi,S)$. 
Because of this we enrich the  spline space $\spline pkn$ by adding 
$$
\mathfrak J^p_k(\Xi,S)= \{f: f|_{(\xi_i,\xi_{i+1})}\in\poly p,\, f\in \smooth {\min \{k, s_i\}}(\xi_{i-1},\xi_{i+1}),\ i=1,\dots,T \}
$$
and obtaining
$$
\brkS pkn(\Xi,S)= \spline pkn + \mathfrak J^p_k(\Xi,S).
$$
Thus $\brkS pkn(\Xi,S)$ is a spline space having varying degree of smoothness at the breakpoints. An example is shown in Figure~\ref{fig:broken}.

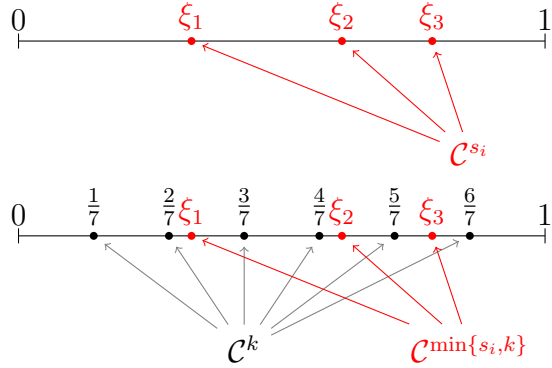
\begin{figure}
\begin{center}
	\begin{tikzpicture}
	\draw (0,0) node[above] {$0$} -- (7,0) node[above] {$1$};
	\draw (0,-.1)--(0,.1)(7,-.1)--(7,.1);
	\node[red] (CR) at (6,-1.5) {$\smooth{s_i}$};

	\foreach \x[count=\i] in {2.3, 4.3, 5.5}
		{
		\fill[red] (\x,0) circle (1.5pt) node (Q\i) {};
		\draw[red,->] (CR)--(Q\i);
		\node[above,red] at (Q\i) {$\xi_\i$};
		}
	\end{tikzpicture}
	
	\begin{tikzpicture}
	\draw (0,0) node[above] {$0$} -- (7,0) node[above] {$1$};
	\draw (0,-.1)--(0,.1)(7,-.1)--(7,.1);
	\node (CK) at (3,-1.5) {$\smooth k$};
	\node[red] (CR) at (6,-1.5) {$\smooth {\min\{s_i,k\}}$};

	\foreach \x[count=\i]  in {1,...,6}
		{
		\fill (\x,0) circle (1.5pt) node (K\i){};
		\draw[gray,->] (CK)--(K\i);
		\node[above] at (K\x) {$\frac \i7$};

		}
	\foreach \x[count=\i] in {2.3, 4.3, 5.5}
		{
		\fill[red] (\x,0) circle (1.5pt) node (Q\i) {};
		\draw[red,->] (CR)--(Q\i);
		\node[above,red] at (Q\i) {$\xi_\i$};
		}
	\end{tikzpicture}
\end{center}
\caption{Above, the breakpoints and corresponding regularities defining $\brkH^{p+1}(\Xi,S)$. Below, those defining $\brkS pk7(\Xi,S)$.}\label{fig:broken}
\end{figure}

Let $\brkC pkn(\Xi,S)$ be the smallest constant such that for all $f\in\brkH^{p+1}(\Xi,S)$ we have
$$
\norm{f- \mathfrak P_{p,k,n} f} \le \brkC pkn(\Xi,S) \norm{\partial^{p+1}f}_\Xi,
$$
where $\mathfrak P_{p,k,n}$ is the orthogonal projection onto $\brkS pkn(\Xi,S)$. 
As in the previous sections we compare $\brkC p{p-1}m(\Xi,S)$ with $\brkC pkn(\Xi,S)$.
In this case it is not always possible to choose $m$ such that $\dim \brkS p{p-1}m(\Xi,S)=\dim\brkS pkn(\Xi,S)$ because an increment of $1$ in $m$ does not necessarily correspond to an increment of $1$ in $\dim \brkS p{p-1}m(\Xi,S)$, e.g., when some of the  breakpoints of $\spline p{p-1}m$ align with the points in $\Xi$.
The dimension of $\brkS pkn(\Xi,S)$ is
$$
\dim \brkS pkn(\Xi,S)= (p-k)(n-1) +p+1 + \sum_{i=1}^T \sigma_i
$$
where
$$\sigma_i=\begin{cases}
p-\min\{k, s_i\}& \xi_i\not\in\{j/n,j=1,\dots,n-1\}\\
\max\{k-s_i, 0\}&\xi_i\in \{j/n,j=1,\dots,n-1\}.
\end{cases}
$$
In particular for $k=p-1$ we have $k\le s_i$ and $k-s_i=(p-s_i)-1$ giving
$$
\dim \brkS p{p-1}m(\Xi,S)= m +p + \sum_{i=1}^T (p-s_i) - \#(M\cap \Xi)
$$
where $M=\{i/m:\,i=1,\dots,m-1 \}$.
Our choice of $m$ is 
\begin{equation}\label{eq:brk-m}
m= (n-1)(p-k)+1+ \sum_{i=1}^T(\sigma_i+s_i-p)
\end{equation}
for which we have
 $$\dim \brkS p{p-1}m(\Xi,S)=\dim \brkS pkn(\Xi,S)-\#(M\cap \Xi)\leq \dim \brkS pkn(\Xi,S).$$

\begin{lemma}
For all $\Xi$ and $S$ we have
\begin{equation}\label{eq:broken-bounds}
\frac{(p+1)!}{(2p+2)!\sqrt{2p+3}} (n+T)^{-p-1}\le \brkC pkn(\Xi,S) \le \const pkn
\end{equation}
\end{lemma}

\begin{proof}
The lower bound is obtained by looking at $k=-1$. In this case $\brkS pkn(\Xi,S)$ is a space of discontinuous piecewise polynomials on the non-uniform partition containing the  intersections $(\xi_i,\xi_{i+1})\cap I_j$. This partition has at most $n+T$ elements, moreover for a given number of elements  the approximation error of $x^{p+1}$ is minimized by the uniform partition. We can thus use $\const p{-1}{n+T}$ as a lower bound.

Next we look at the upper bound.
Given any $f \in \brkH^{p+1}(\Xi,S)$ we can choose a $g\in \mathfrak J^p_k(\Xi,S)$ such that $f-g\in H^{p+1}$ and 
\begin{align*}
\Vert f -\mathfrak P_{p,k,n} f\Vert^2 &=\Vert (f-g) -\mathfrak P_{p,k,n} (f-g) \Vert^2\\
&\le\norm{(f-g) -\proj pkn (f-g)}^2\\
&\le (\const pkn \norm{\partial^{p+1}(f-g)})^2\\
&=(\const pkn)^2\sum_{i=0}^T \norm{\partial^{p+1}(f-g)}^2_{L^2(\xi_i,\xi_{i+1})}\\
&=(\const pkn \norm{\partial^{p+1}f}_\Xi)^2.
\end{align*}
The result then follows by taking the square-root of both sides.
\end{proof}

Similarly to Section \ref{sec:low} we deduce the following result:

\begin{theorem}\label{thm:broken}
 Let $\Xi$ and $S$ be given. If $\rt pk{\bar n}<1$ for some $\bar n$, then for all 
\begin{equation}\label{ineq:2}
n\ge \Big(1  + \frac{T(p-k)-\sum_{i=1}^T (\sigma_i+s_i-p)}{p-k-1} \Big)\bar n - T
\end{equation}
and $m$ as in \eqref{eq:brk-m} we have
$$
\brkC p{p-1}m<\brkC pkn.
$$
\end{theorem}

\begin{proof}
	Reasoning as in Section \ref{sec:comparison-1d} we have
	$$
	\frac {\brkC p{p-1}m}{\brkC pkn}\le \brkR pkn.
	$$
	where
	\begin{align*}
	\brkR pkn&= (\brkB pkn)^{p+1}\sqrt{4q+6}\quad\text{with}\\
	\brkB pkn&=\frac{4}{e\pi} \frac{(n+T) (p+1)}{(p-k)(n-1)+1+\sum_{i=1}^T(\sigma_i+s_i-p)}.
	\end{align*}
Moreover, $$\brkB pkn \le \B pk{\bar n}\quad \Rightarrow\quad \brkR pkn\le \rt pk{\bar n}$$ 
and
 $\brkB pkn \le \B pk{\bar n}$ is equivalent to \eqref{ineq:2}.
\end{proof}

\section{The multivariate case}\label{sec:tens}

In this section we consider the unit hypercube  $\Omega=(0,1)^d$ and a tensor product space
\begin{equation}\label{eq:tensor}
{\bf V}= V_1 \otimes\dots\otimes V_d.
\end{equation}

Let $\Pi_{\bf V}$ be the  $L^2(\Omega)$ projection onto $\bf V$, $\Pi_{i}$ the $L^2(0,1)$ projection onto $V_i$ and $C_{i}$ be the smallest constant in the univariate estimate
$$
\Vert f-\Pi_{i} f\Vert \le C_i \Vert \partial^{q_i+1}f\Vert,\qquad\forall f\in H^{q_i+1}(0,1).
$$

\begin{theorem}\label{thm:multivariate}
For all $f\in L^2(\Omega)$, such that $\partial_i^{q_i+1} f \in L^2(\Omega)$ we have
\begin{equation}\label{ineq:kd}
\norm{f-\Pi_{\bf V} f} \le \sum_{i=1}^d C_i \norm{\partial_i^{q_i+1}f}
\end{equation}
and the result is sharp.
\end{theorem}

\begin{proof}
First of all, we remind that $L^2(\Omega)=L^2(0,1)\otimes\dots\otimes L^2(0,1)$ and that $\Pi_{\bf V}$ factorizes as
$$
\Pi_{\bf V} = \Pi_1\otimes \dots \otimes \Pi_d.
$$
These factors commute as in
 $\Pi_i\otimes \Pi_j=(\Pi_i\otimes \mathrm{I})\circ(\mathrm{I}\otimes \Pi_j)=(\mathrm{I}\otimes \Pi_j)\circ (\Pi_i\otimes \mathrm{I})$ where $\mathrm I$ is in the identity operator.

The theorem is proved by induction on $d$. For $d=1$ it is the definition of $C_i$.
Now suppose the result is true for dimension $d-1$, i.e., that for $\Pi_*=\Pi_1\otimes\dots\otimes\Pi_{d-1}$ we have
$$
\norm{f-\Pi_* f}\le \sum_{i=1}^{d-1}C_i\norm{\partial_i^{q_i+1}f}.
$$
Using the triangle inequality and that $\norm{\mathrm{I}\otimes\Pi_d}=1$ we find
\begin{align*}
\norm{f- \Pi_1\otimes\dots\otimes \Pi_d f}&\le  \norm{f-\mathrm{I}\otimes\Pi_d f} + \norm{\mathrm{I}\otimes\Pi_df- \Pi_*\otimes\Pi_d f}\\
&\le  C_d \norm{\partial_{d}^{q_d+1}f} + \norm{\mathrm{I}\otimes\Pi_d} \norm{f - \Pi_*\otimes \mathrm{I} f}\\
 &\le \sum_{i=1}^d C_i  \norm{\partial_i^{q_i+1}f}.
\end{align*}

To see that the result is sharp we consider $f(x_1,\dots,x_d)=g(x_i)$  and notice that the statement is false for any constant smaller than $C_i$.
\end{proof}

From Theorem \ref{thm:multivariate} we deduce that all conclusions obtained in the univariate comparisons extend to the tensor product setting by considering each direction separately.

\section{Conclusions}

In this paper we have provided a mathematical justification for the numerically observed phenomena that $\smooth {p-1}$ splines of degree $p$, the so-called $k$-\emph{method} in IGA, provide better approximation in degrees of freedom than splines of smoothness $\smooth {-1}$ (DG) and $\smooth {0}$ (FEM). Specifically, we have shown in Section~\ref{sec:comparison-1d} that for sufficiently refined uniform partitions, $\smooth{p-1}$ splines yield better a priori error estimates than $\smooth{-1}$ splines for $p\ge 2$, and $\smooth0$ splines for $p\ge 3$, when approximating functions in the Sobolev space $H^{p+1}$.
For $p=2$ it is an open problem whether $\smooth {1}$ spline spaces provide better a priori error estimates than $\smooth0$ spline spaces of the same dimension. Sharper estimates on the approximation constants could complete the result for this case. Since we have used the lower bound for discontinuous spline approximation also as the lower bound for continuous spline approximation, it seems reasonable to look for an improved lower bound on the approximation constants for $\smooth0$ splines.

It is worth mentioning that the combination of the techniques in Sections~\ref{sec:low} and \ref{sec:broken} allow also the comparison for lower order broken Sobolev spaces. This comparison has not been included.

\section*{Appendix}
\begin{lemma}\label{lem:poly-interp}
Let $p\geq 1$ be any odd number. Then
the interpolation problem: find $g \in \poly p$ such that for all $s=0,2,\ldots,p-1,$
\begin{equation}\label{poly-interp}
\left\{\begin{aligned}
\partial^s g(0)&=a_s\\
\partial^s g(1)&=b_s
\end{aligned}\right .\ 
\end{equation}
admits a solution for all $\{a_s\}$, $\{b_s\}$.
\end{lemma}
\begin{proof}
We proceed by induction on $p$. If $p=1$ then the linear interpolant $g(x)=a_0 + x (b_0-a_0)$ satisfies $g(0)=a_0$ and $g(1)=b_0$ and is a solution.
Now, let $p\geq 3$ be any odd number and assume the result is true for $p-2$. Let $f\in \poly {p-2}$ be the solution of
\begin{equation*}
\partial^s f(0)=a_{s+2},\quad \partial^s f(1)=b_{s+2}\quad s=0,2,\ldots,p-3,
\end{equation*}
which we know exist by the induction hypothesis. We then define the function $g$ by integrating $f$ twice, i.e., 
$$g(x)=cx+d+\int_0^x\int_0^yf(z)dzdx.$$ 
This function then satisfies the cases $s\geq2$ of \eqref{poly-interp} for all $c,d\in\R$. We finish the proof by picking the constants $c$ and $d$ such that the case $s=0$, meaning $g(0)=a_0$ and $g(1)=b_0$, is also satisfied.
\end{proof}

\begin{lemma}\label{lem:poly-interp2}
Let $p\geq 0$ be any even number. Then
the interpolation problem: find $g \in \poly p$ such that for all $s=1,3,\ldots,p-1,$
\begin{equation}\label{poly-interp2}
\left\{\begin{aligned}
\partial^s g(0)&=a_k\\
\partial^s g(1)&=b_k
\end{aligned}\right .\
\end{equation}
admits a solution for all $\{a_s\}$, $\{b_s\}$.
\end{lemma}
\begin{proof}
For $p=0$ there is nothing to prove, and so we consider an even number $p\geq 2$. We then let $f\in\poly{p-1}$ be the solution of
\begin{equation*}
\partial^s f(0)=a_{s-1},\quad \partial^s f(1)=b_{s-1}\quad s=0,2,\ldots,p-2,
\end{equation*}
which we know exist by Lemma \ref{lem:poly-interp}. The function $g(x)=c+\int_0^xf(y)dy$ is then a solution of \eqref{poly-interp2} for any $c\in\R$.
\end{proof}

\section*{Acknowledgements}
The research leading to these results has received funding from the European Research Council under the European Union's Seventh Framework Programme (FP7/2007-2013) / ERC grant agreement 339643.

\bibliography{k-ref}

\providecommand{\bysame}{\leavevmode\hbox to3em{\hrulefill}\thinspace}
\providecommand{\MR}{\relax\ifhmode\unskip\space\fi MR }
\providecommand{\MRhref}[2]{%
  \href{http://www.ams.org/mathscinet-getitem?mr=#1}{#2}
}
\providecommand{\href}[2]{#2}
\begin{thebibliography}{10}

\bibitem{MR2250029}
Y.~Bazilevs, L.~Beir\~ao~da Veiga, J.~A. Cottrell, T.~J.~R. Hughes, and
  G.~Sangalli, \emph{Isogeometric analysis: approximation, stability and error
  estimates for {$h$}-refined meshes}, Math. Models Methods Appl. Sci.
  \textbf{16} (2006), no.~7, 1031--1090.

\bibitem{MR2800710}
L.~Beir\~ao~da Veiga, A.~Buffa, J.~Rivas, and G.~Sangalli, \emph{Some estimates
  for {$h$}-{$p$}-{$k$}-refinement in isogeometric analysis}, Numer. Math.
  \textbf{118} (2011), no.~2, 271--305.

\bibitem{Buffa:14}
L.~Beir\~ao~da Veiga, A.~Buffa, G.~Sangalli, and R.~V\'azquez,
  \emph{Mathematical analysis of variational isogeometric methods}, Acta Numer.
  \textbf{23} (2014), 157--287.

\bibitem{Evans:2009}
J.~A. Evans, Y.~Bazilevs, I.~Babu{\v s}ka, and T.~J.~R. Hughes,
  \emph{{$n$}-widths, sup-infs, and optimality ratios for the {$k$}-version of
  the isogeometric finite element method}, Comput. Methods Appl. Mech. Engrg.
  \textbf{198} (2009), no.~21-26, 1726--1741.

\bibitem{Floater:2017}
M.~S. Floater and E.~Sande, \emph{Optimal spline spaces of higher degree for
  {$L^2$} {$n$}-widths}, J. Approx. Theory \textbf{216} (2017), 1--15.

\bibitem{Floater:2018}
\bysame, \emph{Optimal spline spaces for {$L^2$} {$n$}-width problems with
  boundary conditions}, Constr. Approx. (2018).

\bibitem{Hughes:2005}
T.~J.~R. Hughes, J.~A. Cottrell, and Y.~Bazilevs, \emph{Isogeometric analysis:
  {CAD}, finite elements, {NURBS}, exact geometry and mesh refinement}, Comput.
  Methods Appl. Mech. Engrg. \textbf{194} (2005), no.~39-41, 4135--4195.

\bibitem{Kolmogorov:36}
A.~Kolmogorov, \emph{{\"U}ber die beste {A}nn{\"a}herung von {F}unktionen einer
  gegebenen {F}unktionenklasse}, Ann. of Math. \textbf{37} (1936), 107--110.

\bibitem{Melkman:78}
A.~A. Melkman and C.~A. Micchelli, \emph{Spline spaces are optimal for {$L^2$}
  n-width}, Illinois J. Math. \textbf{22} (1978), 541--564.

\bibitem{Pinkus:85}
A.~Pinkus, \emph{n-{W}idths in approximation theory}, Springer-Verlag, Berlin,
  1985.

\bibitem{Robbins:55}
H.~Robbins, \emph{A remark on stirling's formula}, The American Mathematical
  Monthly \textbf{62} (1955), no.~1, 26--29.

\bibitem{Takacs:2016}
S.~Takacs and T.~Takacs, \emph{Approximation error estimates and inverse
  inequalities for {B}-splines of maximum smoothness}, Math. Models Methods
  Appl. Sci. \textbf{26} (2016), no.~7, 1411--1445.

\end{thebibliography}

\end{document}